\documentclass[12pt]{article}
\usepackage{url}
\usepackage{graphicx}
\usepackage{amsmath}
\usepackage{amsthm}
\usepackage{amssymb}
\usepackage[noline]{algorithm2e}
\usepackage{geometry}
\usepackage{longtable}

\usepackage{pgf,tikz}
\usepackage{tkz-euclide}
\usepackage{tkz-graph}

\usetikzlibrary{calc}
\usetkzobj{all}

\usepackage{pgfplots}
\newlength\figureheight 
\newlength\figurewidth 

\newtheorem{lem}{Lemma}

\newtheorem{thm}{Theorem}
\newtheorem{defn}{Definition}
\newtheorem{rem}{Remark}

\usepackage{overpic}

\title{Optimization on the biorthogonal manifold}

\author{Klaus Glashoff\\
\small{USI Lugano}\\
\small{University of Hamburg}
 \and 
 Michael M. Bronstein\\
\small{USI Lugano}\\
\small{Tel Aviv University}\\
\small{Intel Perceptual Computing}\\
}

\begin{document}
\maketitle

\abstract{
In this paper, we consider optimization problems w.r.t. to pairs of orthogonal matrices $XY = I$. Problems of this form arise in several applications such as finding shape correspondence in computer graphics. 
We show that the space of such matrices is a Riemannian manifold, which we call the {\em biorthogonal manifold}. 
To our knowledge, this manifold has not been studied before. We give expressions of tangent space projection, exponential map, and retraction operators of the biorthogonal manifold, and discuss their numerical implementation. 
}

\section{Introduction}
 \subsection{Manifold optimization}

 The term {\em manifold-} or {\em manifold-constrained optimization} refers to a class of problems of the form 
\begin{eqnarray}
\min_{X \in \mathcal{M}} f(X), 
\label{eq:prob0}
\end{eqnarray}
where $f$ is a smooth real-valued function, $X$ is an $m\times n$ real matrix, and $\mathcal{M}$ is some Riemannian submanifold of $\mathbb{R}^{m\times n}$. 
%
%
%
The main idea of manifold optimization is to treat the objective as a function $f:\mathcal{M} \rightarrow \mathbb{R}$  defined on the manifold,  and perform descent on the manifold itself rather than in the ambient Euclidean space. 

A key difference between classical and manifold optimization is that manifolds do not have a global vector space structure and are only locally homeomorphic to a Euclidean space, referred to as as the {\em tangent space}. 
The intrinsic (Riemannian) gradient $\nabla_{\mathcal{M}}f(X)$ of $f$ at point $X$ on a manifold is a vector in the tangent space $T_{X}\mathcal{M}$ that can be obtained by projecting the standard (Euclidean) gradient  $\nabla f(X)$ onto $T_{X}\mathcal{M}$ by means of a {\em projection} operator $P_X$ (see Figure~\ref{fig:manopt}). 
A step along the intrinsic gradient direction is performed in the tangent space. In order to obtain the next iterate, the point in the tangent plane is mapped back to the manifold by means of a {\em retraction} operator $R_X$, which is typically an approximation of the exponential map. 
%
%

\begin{figure}[h!]
\centering
\begin{overpic}
	[width=0.5\linewidth]{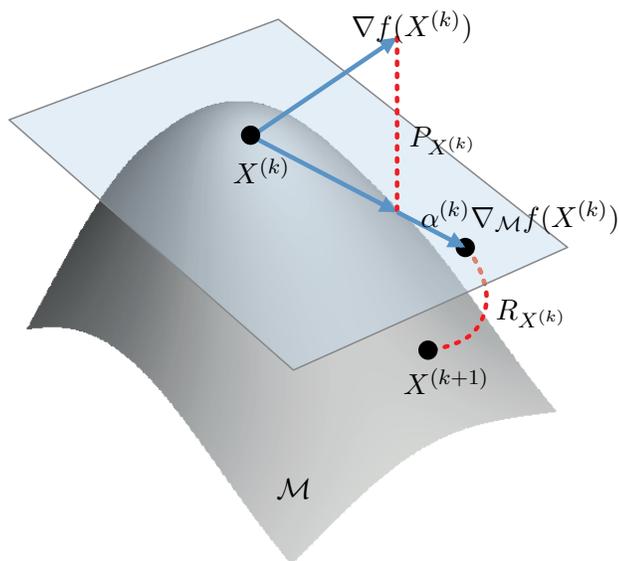}
	\put(39.5,69.5){\small$X^{(k)}$}	
	\put(60,95){\small$\nabla f(X^{(k)})$}	
	\put(70,76){\small$P_{X^{(k)}}$}	
	\put(72,62){\small$\alpha^{(k)}\nabla_{\mathcal{M}}f({X^{(k)}})$}	
	\put(85,46){\small$R_{X^{(k)}}$}	
	\put(69,33){\small$X^{(k+1)}$}	
	\put(47,15){\small$\mathcal{M}$}	
\end{overpic}\vspace{2mm}
\caption{\label{fig:manopt} \small Illustration of a typical step of a first-order manifold optimization algorithm: the Euclidean gradient $\nabla f$ is projected on the tangent space by means of a projection operator $P_X$, producing a Riemannian gradient $\nabla_\mathcal{M} f$. A line search is performed along the gradient direction to determine the step size $\alpha$. Finally, the result is mapped back to the manifold by means of a retraction operator $R_X$.  }
\end{figure}

\RestyleAlgo{boxed}
\DecMargin{0.75em}
\begin{algorithm}[H]

\Repeat{convergence}{
1. Compute the extrinsic gradient $\nabla f(X^{(k)})$ \\
2. {\em Projection:} 
$\nabla_{\mathcal{M}} f(X^{(k)}) = P_{X^{(k)}} (\nabla f(X^{(k)}))$ \\
3. Compute the step size $\alpha^{(k)}$ along the descent direction \\
4. {\em Retraction:} $X^{(k+1)} = R_{X^{(k)}}(- \alpha^{(k)} \nabla_{\mathcal{M}} f(X^{(k)}))$ \\
}
\caption{\label{algo:manopt} \small Conceptual algorithm for optimization on manifold $\mathcal{M}$.}
\end{algorithm}

\vspace{5mm}

A conceptual gradient descent-like manifold optimization is presented in Algorithm~\ref{algo:manopt}. 
Besides the standard ingredients of first-order optimization methods (computation of the gradient and line search), it contains two special steps: projection on the tangent space and retraction. For many manifolds, these operations have a closed-form expression \cite{Absil2008}.

First manifold optimization algorithms computed the exponential map to perform the mapping from the tangent space to the manifold \cite{luenberger1972gradient,Gabay1985,helmke2012optimization,smith1994optimization,mahony1994optimization}. In many cases, however, such a computation is expensive. 
The powerful concept of a retraction, a computationally-efficient approximation of the exponential map, has entered the discussion around the year 2000 (see \cite{Absil2008} and a short historical overview in \cite{Absil2012}). 
Since then, efficient retraction operators have been derived for several manifolds. 

 %


 \subsection{Problems on the biorthogonal manifold}
 
 In this paper, we are interested in minimizing functions of the form $f(X,X^{-1})$ for non-singular square matrices, which can be cast as 
 $$
 \min_{X,Y} f(X,Y) \hspace{5mm} \mathrm{s.t.} \hspace{5mm} XY = I. 
 $$
Our main interest is in the set of pairs of biorthogonal matrices $(X,Y) : XY = I$, which we show to be a manifold in the next section.  
Such problems arise in computer graphics in relation to computation of functional maps between non-rigid shapes \cite{Eynard2016}, as discussed in Section~\ref{sec:exp}. 
%

\section{The biorthogonal manifold $BO(n)$}
 \subsection{Definition of $BO(n)$}
In the following, we use $M(n)=\mathbb{R}^{n\times n}$ to denote the space of real $n\times n$ matrices. The {general linear group} $GL(n)$ of invertible real $n\times n$ matrices is an $n^2$-dimensional submanifold of $M(n)$. It has a group structure w.r.t. the standard matrix multiplication operation, with identity element $I$ and inverse element $X^{-1}$.

We consider the $2n^2$-dimensional product manifold $\mathcal{G}(n)=GL(n) \times GL(n)$ of pairs of invertible matrices, which 
form a \emph{Lie group} (i.e., a finite dimensional differentiable manifold as well as a group with  differentiable operations \cite{Lee2012}) w.r.t. the pair product operation  
$$(X_1,X_2) \star (Y_1,Y_2)=(X_1Y_1,Y_2X_2);$$ 
note that the order of matrices in the first and second component of the product is different. 

\begin{defn}
Let $BO(n) = \{ (X ,Y) : XY=I \} \subset \mathcal{G}(n)$ be the subset of orthogonal pairs of invertible $n\times n$ real matrices.\footnote{An equivalent definition is $BO(n)=\left\{(X,X^{-1}): X\in GL(n)\right\}$. }
\end{defn}
\begin{lem} 
$BO(n)$ is a $n^2$-dimensional submanifold of  $\mathcal{G}(n)$.
\end{lem}
\begin{proof}
$BO(n) = F^{-1}(0)$, where $F: \mathcal{G}(n) \rightarrow M(n)$ is defined by $F(X,Y)=XY-I$. 
The map $dF_{(X_0,Y_0)}(X,Y)=X_0 Y+X Y_0$ 
is surjective because the linear system 
$$X_0Y+XY_0=Z$$ 
of $n^2$ linear equations in  $2n^2$ unknowns 
has a solution $(\hat{X},\hat{Y})$ for any $Z\in M(n)$. 
In particular, this system has the special solution $\hat{X}=0, \hat{Y}=X_0^{-1}Z$. This implies that $dF_{(X_0,Y_0)}$ is a submersion, and so $BO(n)=F^{-1}(0)$ is a manifold of dimension $2n^2-n^2=n^2$ by virtue of the Preimage Theorem  (\cite{Pollak1974}, p. 21). 
\end{proof}

\begin{rem}
By an alternative argument, the map $X \rightarrow(X,X^{-1})$ is an \emph{embedding} of $GL(n)$ into $\mathcal{G}(n)$, and so its image $BO(n)$ is a diffeomorphic image of $GL(n)$. 
\end{rem}

\begin{rem}
\label{rem:Lie}
 It is easy to check that $BO(n)$ is a \emph{Lie subgroup} of the Lie group $\mathcal{G}(n)$, as follows: for any $(X,X^{-1}), (Y,Y^{-1})\in BO(n)$, we have closure w.r.t. to the pair product operation,
 $$(X,X^{-1})\star (Y,Y^{-1}) = (XY,Y^{-1}X^{-1})=(XY,(XY)^{-1})\in BO(n);$$
 inverse element $(X,X^{-1})^{-1}=(X^{-1},X)\in BO(n)$ satisfying $$(X,X^{-1})^{-1}\star(X,X^{-1}) = (X^{-1}X,X^{-1}X) = (I,I),$$ and identity element 
 $(I,I)\star(X,X^{-1}) = (X,X^{-1})$. 
 Finally, the mapping $((X,X^{-1}),(Y,Y^{-1})) \rightarrow (X^{-1}Y,(X^{-1}Y)^{-1})$ is smooth. 
\end{rem}

  \subsection{Tangent space and its projection operator}
  The tangent space of the manifold $BO(n)$ at a point $(X_0,Y_0)$ is given by  
\begin{equation}
\label{equ:tangentspace}
T_{(X_0,Y_0)}BO(n)=\left\{ ( X_0Y+XY_0=0 : X,Y\in M(n) \right\}.
\end{equation}
 This can be seen by performing a first-order approximation of the equation $X_0Y_0=I$ as follows: if $(X,Y)$ is a vector in the tangent space, then $$(X_0+X)(Y_0+Y)=I+o(t)$$ for $t\rightarrow 0$, and this leads directly to (\ref{equ:tangentspace}).  
  This implies that the tangent space at the identity $(I,I)$ of $BO(n)$ is $T_{(I,I)}BO(n)=\left\{ ( X,-X) : X \in M(n) \right\}$.
  
A crucial component of optimization on the biorthogonal manifold is the  \emph{projection operator}
    $P_{(X_0,Y_0)}: \mathcal{G}(n)  \rightarrow  T_{(X_0,Y_0)}BO(n)$.
%
The image $(\hat{X},\hat{Y})=P_{(X_0,Y_0)}(\Phi,\Psi)$ of a given $(\Phi,\Psi)\in \mathcal{G}(n)$ is the unique solution of the quadratic optimization problem 
\begin{equation}
\label{problem:Sylvester}
\min_{X,Y} {||X-\Phi||^2_\mathrm{F}+||Y-\Psi||^2_\mathrm{F} } \,\,\, \mathrm{s.t.} \,\,\, X_0Y+XY_0=0
\end{equation}
By means of the coordinate transformation $\tilde{X} = (X-\Phi)$,  $\tilde{Y}=(Y-\Psi)$ we obtain the problem of finding the minimum-norm solution of the linear system 
\begin{equation}
\label{equ:Sylvester}
 X_0\tilde{Y}+\tilde{X}Y_0=-X_0\Psi-\Phi Y_0.
\end{equation} 

Minimizing $||\tilde{X}||^2_\mathrm{F}+||\tilde{Y}||^2_\mathrm{F}$ under the constraint of this so called  \emph{Generalized Sylvester} type equation (\ref{equ:Sylvester}) is a well known problem of numerical analysis, a solution method for which was given by Stewart \cite{Stewart1992} and, more generally, in  \cite{ShimChen2003}. As the set of matrices $(X,Y)\in \mathcal{G}(n)$ satisfying the constraints of (\ref{problem:Sylvester}) is a nonempty linear subspace, there is a unique solution to the quadratic optimization problem (we defer the details to Section~\ref{sec:numerics}).

  \subsection{Exponential function and retraction}
  \label{retrexp}
Another important ingredient of modern optimization algorithms on manifolds is the \emph{retraction} mapping from the tangent bundle onto the manifold \cite{Absil2008}. A function $R_X(U)$ defined on the tangent bundle of a manifold is a retraction if it approximates the exponential map of the manifold in the following sense: 
$$R_X(tU)=\exp(X,tU)+o(t),$$
for $t\rightarrow 0$. 
Of course, the exponential map itself could be used to perform retraction, and this has been the first choice in the beginnings of the theory of optimization on manifolds (see  \cite{luenberger1972gradient}, \cite{Gabay1985}, \cite{Absil2008}). Later, other computationally more effective retractions have been developed for special manifolds. One example is the so-called \emph{Cayley transformation}, used especially for the minimization of functions on the orthogonal group $O(n)$ and the special orthogonal group $SO(n)$ \cite{Edelman1998,yamada2003orthogonal,plumbley2004lie,Wen2013}.
Specific retractions for various manifolds have been effectively implemented in the Manopt MATLAB toolbox \cite{Boumal2014}. 

%


The following theorem provides for a general way of constructing retraction operators. 
\begin{thm}[Projective retraction] Let $\mathcal{M}$ be a submanifold of an Euclidean space  $\mathcal{E}$, which is $C^k$ around $X_0 \in \mathcal{M}$. Let $P_\mathcal{M}$ denote the projection operator from $\mathcal{E}$ to $\mathcal{M}$. Then, the function 
$$
R: (X,U) \rightarrow P_\mathcal{M} (X+U)$$ is a retraction around $X_0$.
\end{thm}
\begin{proof}
We refer the reader to \cite{Absil2012} for the background theory and proof. 
\end{proof}


As we have seen in Remark \ref{rem:Lie}, our biorthogonal manifold $BO(n)$ is a Lie subgroup of $\mathcal{G}(n)$. 
As the exponential map on $GL(n)$ is the ordinary matrix exponential 
$$e^{U}=I+U+\frac{1}{2!}U^2+\frac{1}{3!}U^3+ \ldots,$$ 
the exponential map on $\mathcal{G}(n)$ is given by
\begin{equation}
\label{equ:exp}
\exp(U,V)=(e^U,e^V)
\end{equation}
for $(U,V)\in T_{(X_0,Y_0)}BO(n)$

\begin{thm}
\label{them:exp1}
The exponential map on $BO(n) =\left\{(X,X^{-1}):X\in GL(n)\right\}$  
is given by $\exp(X)=(e^{X},e^{-X})$. 
\end{thm}
\begin{proof}
By Proposition 15.19 of \cite{Lee2012}, the exponential map on the subgroup $BO(n)$ is the restriction of the exponential on $\mathcal{G}(n)$ to the tangent space $T_{(I,I)}BO(n)=\left\{ ( X,-X) : X \in M(n) \right\}$. This implies that the exponential on $BO(n)$ can  be expressed as $\exp(X)=(e^{X},e^{-X})$.
\end{proof}

\section{Numerical implementation}
\label{sec:numerics}

The numerical implementation of optimization methods on the manifold $BO(n)$ requires two ingredients, namely, an algorithm for projecting vectors onto the tangent space and a retraction operator. For the latter, we take the exponential function defined in Theorem~\ref{them:exp1}. 
For matrices of moderate size ($n$ ranging between $10^2-10^3$), such an exponential map can be efficiently computed. This applies in particular to our main example in the results section. 


The computation of the projection $P_{(X_0,Y_0)}$ on the tangent space 
%
can be done by standard numerical linear algebra, utilizing two singular value decompositions for the solution of a \emph{Generalized Sylvester Equation} (GSE). This method goes back to \cite{Stewart1992} and has been generalized since then to a wider class of GSE's by
\cite{ShimChen2003}\footnote{The \emph{Sylvester} equation is a linear equation of the form $AX+XB=C (\cite{BartelsStewart1972}).$ In \cite{Stewart1992}, the \emph{GSE} $AX+YB=C$ is treated for singular and even rectangular matrices, too. The setting of \cite{ShimChen2003} is even more general, covering equations of the form $AXB^*+CYD^*=E$.}.
For sake of completeness we will show how to construct the solution of this equation in detail. 

\subsection{Numerical computation of $P_{(X_0,Y_0)}$}

Given the matrices $(X_0,Y_0) \in BO(n)$ and $\Phi,\Psi \in GL(n)$, our goal is to solve 
\begin{equation}
\label{problem:Sylvester2}
\min_{X,Y \in M(n)} {||X||^2_\mathrm{F} + ||Y||^2_\mathrm{F} \,\,\, \mathrm{s.t.} \,\,\, X_0Y+XY_0=-X_0\Psi-\Phi Y_0.}
\end{equation}
Let $X_0=U_xS_xV_x^\top$  and $Y_0=U_yS_yV_y^\top$ denote the singular value decompositions of $X_0$ and $Y_0$, where $S_x=\mathrm{diag}(\alpha_1, \hdots , \alpha_n)$ and $S_y=\mathrm{diag}(\beta_1, \hdots ,\beta_n)$ and the singular values $\alpha_i, \beta_i > 0$ since $X_0$ and $Y_0$ are invertible. 
Let us define $C=-X_0\Psi-\Phi Y_0$. Then, the linear system 
$$X_0Y+XY_0=-X_0\Psi-\Phi Y_0$$ can be cast as
\begin{equation}
\label{eq:linear2}
U_xS_xV_x^TY+XU_yS_yV_y^T=C
\end{equation}
which, by the orthogonality of $U_x$ and $V_y$, transforms into
\begin{equation}
\label{eq:linear3}
S_x\hat{Y}+\hat{X}S_y=\hat{C},
\end{equation}
where 
$\hat{X}=U_x^\top XU_y$, $\hat{Y}=V_x^\top YV_y$, and $\hat{C}=U_x^\top CV_y$. 
These orthogonal transformations do not change the value of the cost function of problem (\ref{problem:Sylvester2}), thus solving~  (\ref{problem:Sylvester2}) is equivalent to 
\begin{equation}
\label{problem:Sylvester2_}
\min_{\hat{X},\hat{Y} \in M(n)} {||\hat{X}||^2_\mathrm{F}+||\hat{Y}||^2_\mathrm{F} \,\,\, \mathrm{s.t.} \,\,\, S_x\hat{Y}+\hat{X}S_y=\hat{C}.}
\end{equation}

The diagonality of $S_x$ and $S_y$ decouples problem~(\ref{problem:Sylvester2_}) into $n^2$ independent optimization problems
%
\begin{equation}
\min_{\hat{x}_{ij},\hat{y}_{ij}}{\hat{x}_{ij}^2+\hat{y}_{ij}^2 } \,\,\, \mathrm{s.t.} \,\,\, \alpha_i \hat{x}_{ij} +\hat{y}_{ij}\beta_j=\hat{c}_{ij}
\end{equation}
The solution of this scalar problem is given by
\begin{equation}
\label{equ:solution}
\hat{x}_{ij}=\frac{\hat{c}_{ij}\alpha_{i}}{\alpha_i^2+\beta_j^2}, \,\,\,\,\, \hat{y}_{ij}=\frac{\hat{c}_{ij}\alpha_{j}}{\alpha_i^2+\beta_j^2}.
\end{equation}
 Then, the solution $(X,Y)$ is obtained as 
 $X = U_x\hat{X}U_y^\top$ and  $Y = V_x\hat{Y}V_y^\top$.
 %



\RestyleAlgo{boxed}
\DecMargin{0.75em}
\begin{algorithm}[H]

\SetKwInOut{Input}{input} 
\SetKwInOut{Output}{output} 

\Input{$n\times n$ matrices $X_0,Y_0,\Phi,\Psi$}

\Output{Minimum-norm-solution of $X_0Y+XY_0=-X_0\Psi-\Phi Y_0$}

1. Compute the SVD $ X_0=U_xS_xV_x^\top$ and $Y_0=U_yS_yV_y^\top$ and let $\alpha=\mathrm{diag}(S_x), \beta=\mathrm{diag}(S_y)$ be the vectors of singular values. 

2. Let $C=-X_0\Psi-\Phi Y_0$ and $\hat{C}=U_x^\top CV_y$. 

3. Compute $\hat{X}=({\hat{\hat{y}}_{ij}}),\hat{Y}=({\hat{\hat{x}}_{ij}})$ by means of (\ref{equ:solution}).

4. Set $X=U_x\hat{X}U_y^\top$ and $Y=V_x\hat{Y}V_y^\top$.

\caption{\label{algo:project} \small Algorithm for projection on the tangent space of $BO(n)$.}
\end{algorithm}

\section{Examples}
\label{sec:exp}

We implemented the manifold $BO(n)$ of our method as \emph{biorthogonalfactory()} in Manopt, and, as numerical algorithm, we chose the conjugate gradient solver contained in Manopt.

 \subsection{Random matrices}
 To study the behavior of our optimization, we first did numerical tests with random matrices. We consider the following simple model problem:
\begin{equation}
\label{problem:modelproblem}
\min_{XY=I} {\|X-\Phi\|^2_F+\|Y-\Psi\|^2_F }
\end{equation}
where $\Phi, \Psi$ are given matrices  in $M(n)$. 
%

We compared our manifold method with the penalty method,   
\begin{equation}
\label{problem:penaltyproblem}
\min_{X,Y\in M(n)} {\|X-\Phi\|^2_F+\|Y-\Psi\|^2_F+\alpha\|XY-I\|^2_F },
\end{equation}
where $\alpha > 0$ is some parameter. Note that such a formulation in general does not guarantee a feasible solution, but only an approximately feasible one. 
We used a conjugate gradient solver on the Euclidean space implemented in Manopt to numerically solve the problem (\ref{problem:penaltyproblem}).



\begin{figure}[!ht]
\begin{center}
\figurewidth=0.8\linewidth
\figureheight=7.5cm
%
%
\begin{tikzpicture}

\begin{axis}[%
width=\figurewidth,
height=\figureheight,
at={(1.011in,0.642in)},
scale only axis,
xmin=0,
xmax=100,
xlabel={\small Iteration number},
ymin=2000,
ymax=10000,
ylabel={\small Cost function},
axis background/.style={fill=white},
legend style={legend cell align=left,align=left,draw=white!15!black}
]
\addplot [color=red,solid,line width=3pt]
  table[row sep=crcr]{%
0	9917.56571357361\\
1	9758.39415360508\\
2	9451.74406440742\\
3	8882.52978290605\\
4	7900.90921722649\\
5	6435.38162750073\\
6	4775.31466354163\\
7	3643.89569647558\\
8	3312.21623653891\\
9	3196.76065691419\\
10	3137.89593802091\\
11	3107.96034715347\\
12	3081.24250619653\\
13	3068.66293332801\\
14	3059.56311430666\\
15	3053.36657434988\\
16	3050.90390650703\\
17	3049.55400668994\\
18	3048.17111424959\\
19	3046.87390462842\\
20	3045.75296845682\\
21	3045.41285877414\\
22	3045.22432237285\\
23	3044.97460625751\\
24	3044.53225215877\\
25	3044.14326933016\\
26	3043.94907196138\\
27	3043.79458376674\\
28	3043.59985564981\\
29	3043.49222191569\\
30	3043.43239152848\\
31	3043.40250019285\\
32	3043.36590806387\\
33	3043.30773587751\\
34	3043.21052124033\\
35	3043.13133992418\\
36	3043.08367062601\\
37	3043.06582991756\\
38	3043.05623882854\\
39	3043.04114376615\\
40	3043.0189636861\\
41	3042.95382912359\\
42	3042.92967724814\\
43	3042.90343079345\\
44	3042.87706082788\\
45	3042.85286460258\\
46	3042.84036437008\\
47	3042.82980443657\\
48	3042.82227997874\\
49	3042.81603183944\\
50	3042.81174423468\\
51	3042.8083214018\\
52	3042.80495167782\\
53	3042.80190850542\\
54	3042.80053765741\\
55	3042.79995054965\\
56	3042.79973161329\\
57	3042.79951621459\\
58	3042.79939734739\\
59	3042.79918921597\\
60	3042.7990668691\\
61	3042.79879034702\\
62	3042.79864507906\\
63	3042.79859670568\\
64	3042.79853335517\\
65	3042.79848978832\\
66	3042.79846168887\\
67	3042.79842633871\\
68	3042.79840492409\\
69	3042.79834048636\\
70	3042.79829820254\\
71	3042.79827363356\\
72	3042.79825293968\\
73	3042.79824360915\\
74	3042.79823807849\\
75	3042.79823484581\\
76	3042.79823343265\\
77	3042.79823276818\\
78	3042.79823224738\\
79	3042.79823160529\\
80	3042.7982307795\\
81	3042.79822988743\\
82	3042.79822932858\\
83	3042.79822881943\\
84	3042.79822856008\\
85	3042.798228363\\
86	3042.79822810368\\
87	3042.79822786923\\
88	3042.79822762055\\
89	3042.79822754249\\
90	3042.79822748325\\
91	3042.79822744622\\
92	3042.79822741646\\
93	3042.79822740524\\
94	3042.79822738511\\
95	3042.79822735958\\
96	3042.79822734573\\
97	3042.79822733908\\
98	3042.79822733186\\
99	3042.7982273274\\
100	3042.79822732416\\
};
\addlegendentry{\footnotesize Biorthogonal};

\addplot [color=blue,dashed,line width=1.5pt]
  table[row sep=crcr]{%
0	9917.56571357361\\
1	8018.66758525829\\
2	8018.66758525829\\
3	6879.27938825254\\
4	6309.32976358269\\
5	5343.57232372038\\
6	5343.57232372038\\
7	4827.72628360665\\
8	4419.79115693239\\
9	4419.79115693239\\
10	4088.77438433644\\
11	3890.21246576058\\
12	3634.83496791174\\
13	3485.90518941223\\
14	3327.43404937418\\
15	3327.43404937418\\
16	3284.77850390466\\
17	3233.07226554249\\
18	3161.17716993426\\
19	3161.17716993426\\
20	3115.13731360363\\
21	3089.26970817542\\
22	3067.98288180244\\
23	3067.98288180244\\
24	3055.35267977547\\
25	3043.85752913654\\
26	3040.34818480354\\
27	3040.34818480354\\
28	3031.1342308329\\
29	3023.17937862976\\
30	3019.79363829382\\
31	3012.92312890806\\
32	3005.54010627227\\
33	3004.14098749424\\
34	2992.84141000735\\
35	2985.84554354724\\
36	2982.63055063053\\
37	2981.6270467113\\
38	2975.95950565322\\
39	2974.90618558852\\
40	2974.90618558852\\
41	2971.45899053759\\
42	2967.72919868003\\
43	2966.36462069721\\
44	2966.36462069721\\
45	2962.92982681247\\
46	2959.42586380633\\
47	2957.85613766472\\
48	2957.85613766472\\
49	2954.61597787701\\
50	2951.06431726834\\
51	2949.34166839587\\
52	2949.34166839587\\
53	2946.42846079966\\
54	2943.12997569265\\
55	2941.09600683765\\
56	2941.09600683765\\
57	2938.27712913051\\
58	2934.91376481519\\
59	2933.36184076428\\
60	2933.36184076428\\
61	2930.29803615708\\
62	2927.54495707489\\
63	2925.45712742246\\
64	2923.17744462494\\
65	2920.99781759402\\
66	2917.54812872601\\
67	2915.46985125647\\
68	2912.61037741793\\
69	2910.509091752\\
70	2908.05902637632\\
71	2905.92949423039\\
72	2903.05360037123\\
73	2900.80022696269\\
74	2897.8062506314\\
75	2896.11486408963\\
76	2893.04270703487\\
77	2891.27182334386\\
78	2888.8695874692\\
79	2886.84041686565\\
80	2884.39171623212\\
81	2882.29026117757\\
82	2880.12541163186\\
83	2878.04490498081\\
84	2875.98314091903\\
85	2874.08304353395\\
86	2872.15049761577\\
87	2869.8551132913\\
88	2868.10405524347\\
89	2865.85270554297\\
90	2864.41177179321\\
91	2861.94387141484\\
92	2860.32254346643\\
93	2858.13345524981\\
94	2856.83449248642\\
95	2854.65467847651\\
96	2853.28295682004\\
97	2851.13804514017\\
98	2849.78051049358\\
99	2847.64682843139\\
100	2846.3051976824\\
};
\addlegendentry{\footnotesize Penalty $\alpha=10$};

\addplot [color=blue,dashdotted,line width=1.5pt]
  table[row sep=crcr]{%
0	9917.56571357361\\
1	8611.95938343465\\
2	8611.95938343465\\
3	8611.95938343465\\
4	7491.37428140101\\
5	7121.78952609912\\
6	6835.22027257073\\
7	6203.49799007504\\
8	6203.49799007504\\
9	5915.29383576361\\
10	5640.01571140837\\
11	5640.01571140837\\
12	5394.09078883791\\
13	5131.84247297201\\
14	5131.84247297201\\
15	4906.90163514036\\
16	4714.35255335299\\
17	4586.55330112804\\
18	4586.55330112804\\
19	4463.55821166803\\
20	4282.81555599375\\
21	4174.78437011036\\
22	4174.78437011036\\
23	4081.30822126522\\
24	3928.57918684974\\
25	3833.44934130689\\
26	3833.44934130689\\
27	3765.77655678266\\
28	3765.77655678266\\
29	3712.36140756968\\
30	3657.87294631924\\
31	3563.82700154612\\
32	3508.60424460832\\
33	3508.60424460832\\
34	3442.68592044231\\
35	3401.32043229642\\
36	3338.18065262567\\
37	3338.18065262567\\
38	3300.63906040506\\
39	3275.73735378161\\
40	3247.71434343751\\
41	3247.71434343751\\
42	3235.13639046829\\
43	3213.09173226711\\
44	3213.09173226711\\
45	3185.82674485795\\
46	3176.87146500573\\
47	3144.02333412706\\
48	3144.02333412706\\
49	3129.76010110471\\
50	3122.54062074611\\
51	3105.45886354745\\
52	3105.45886354745\\
53	3102.53789283821\\
54	3097.45970086945\\
55	3092.35715038664\\
56	3081.1976584009\\
57	3081.1976584009\\
58	3079.16434798249\\
59	3075.48345284116\\
60	3073.71565128401\\
61	3064.65057920221\\
62	3063.08690555488\\
63	3060.91642450704\\
64	3056.10564425088\\
65	3054.95246795938\\
66	3052.13011957322\\
67	3052.13011957322\\
68	3051.32525655539\\
69	3050.12992150135\\
70	3048.43058169732\\
71	3048.43058169732\\
72	3047.83250423913\\
73	3047.21750909052\\
74	3046.80388142871\\
75	3046.80388142871\\
76	3046.31883195695\\
77	3045.64910139477\\
78	3045.19650718333\\
79	3045.19650718333\\
80	3044.8646997526\\
81	3044.2545582837\\
82	3043.77437164796\\
83	3043.77437164796\\
84	3043.516968278\\
85	3042.86610596091\\
86	3042.2905069463\\
87	3041.39105162537\\
88	3041.39105162537\\
89	3041.17409707822\\
90	3040.81055652395\\
91	3040.09431551766\\
92	3039.69815857827\\
93	3039.69815857827\\
94	3039.37945903445\\
95	3039.23487757443\\
96	3038.8804810978\\
97	3038.8804810978\\
98	3038.81717522706\\
99	3038.68267595695\\
100	3038.5218923708\\
};
\addlegendentry{\footnotesize Penalty $\alpha=10^2$};

\addplot [color=blue,dotted,line width=1.5pt]
  table[row sep=crcr]{%
0	9917.56571357361\\
1	9917.56571357361\\
2	9556.5234220707\\
3	9556.5234220707\\
4	9089.72373806422\\
5	8816.27213846775\\
6	8738.90298713069\\
7	8738.90298713069\\
8	8329.27819190123\\
9	8127.41840933414\\
10	8127.41840933414\\
11	7937.45667220897\\
12	7747.19330664522\\
13	7672.4957465841\\
14	7393.38083399017\\
15	7126.78697406283\\
16	6861.0322810238\\
17	6563.12678584042\\
18	6338.95499709854\\
19	6076.26874185901\\
20	6076.26874185901\\
21	5931.00576851011\\
22	5931.00576851011\\
23	5789.68243872297\\
24	5658.43191749687\\
25	5564.40695803603\\
26	5395.7278572234\\
27	5278.17058592976\\
28	5278.17058592976\\
29	5141.49950323859\\
30	5029.31104447604\\
31	5029.31104447604\\
32	4913.32274904758\\
33	4826.62379118194\\
34	4789.32582757629\\
35	4666.43417478123\\
36	4666.43417478123\\
37	4558.56308269358\\
38	4558.56308269358\\
39	4511.17977167314\\
40	4464.24034292966\\
41	4412.72096579495\\
42	4374.04016849752\\
43	4357.97049758736\\
44	4194.13236654654\\
45	4147.01658538347\\
46	4069.05730292117\\
47	3999.7413378507\\
48	3999.7413378507\\
49	3965.0797884443\\
50	3897.35533705196\\
51	3838.04291952781\\
52	3838.04291952781\\
53	3804.0870535767\\
54	3755.97470051131\\
55	3713.71121400132\\
56	3670.76416979134\\
57	3625.75954579372\\
58	3591.00513263099\\
59	3549.47189691139\\
60	3533.53730804109\\
61	3483.47357021235\\
62	3483.47357021235\\
63	3461.35494254205\\
64	3440.60705153127\\
65	3440.60705153127\\
66	3418.66240715802\\
67	3398.98179650402\\
68	3385.62133657706\\
69	3385.62133657706\\
70	3375.62422722178\\
71	3357.15566764805\\
72	3343.05307450182\\
73	3307.97213576627\\
74	3307.97213576627\\
75	3292.22294999929\\
76	3280.74024756266\\
77	3280.74024756266\\
78	3263.4719780858\\
79	3255.20288660747\\
80	3244.05850999508\\
81	3232.63311654302\\
82	3225.57310348728\\
83	3212.797447674\\
84	3205.71030195598\\
85	3199.21164165516\\
86	3199.21164165516\\
87	3182.61153215873\\
88	3177.55735182226\\
89	3169.76927404839\\
90	3160.84443372045\\
91	3160.84443372045\\
92	3158.57185471726\\
93	3153.74707147073\\
94	3148.47882238267\\
95	3140.31583738729\\
96	3140.31583738729\\
97	3138.29699391615\\
98	3134.62581103705\\
99	3130.36023167313\\
100	3122.8141399747\\
};
\addlegendentry{\footnotesize Penalty $\alpha=10^3$};

\end{axis}
\end{tikzpicture}%
\caption{\label{fig:rand} \small Comparison of penalty and biorthogonal manifold optimization methods on the random matrices experiment. 
}
\end{center}
\end{figure}
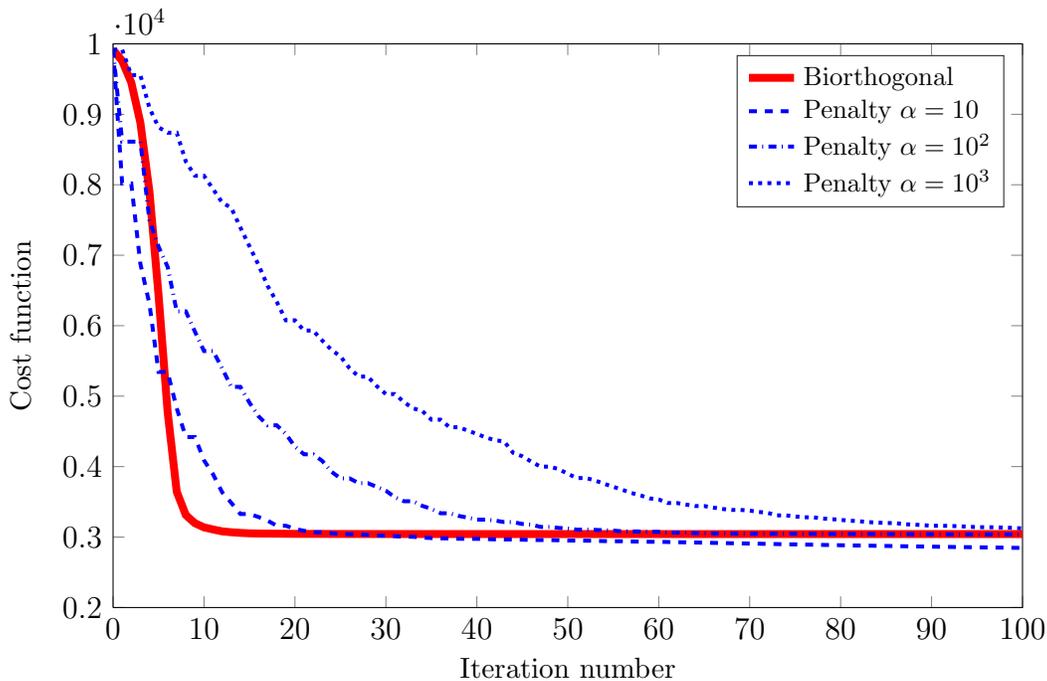



In our first experiment we choose choose $n=100$, i.e., the variables were matrices of order $100\times 100$.
The time per iteration was around $0.18$ \emph{sec} for both our biorthogonal manifold method as well as for the penalty method with $\alpha=100$. But, as can be seen by the figure, only about $12$ iterations were necessary with the biorthogonal method, while for the penalty method  $60$ iterations were necessary to obtain a comparable reduction of the cost function. This is typical for experiments in all dimensions.

In addition, the constraint error $\|XY-I\|_F$ is magnitudes bigger in the penalty method than in the biorthogonal method (where this error is practically zero). By enlarging the penalty parameter, this error can be made smaller, but only to a certain extent, because any minimization algorithm will get into trouble for large $\alpha$.

 \subsection{Functional maps}
 
 In computer graphics, one of the central problems arising in numerous applications, is finding intrinsic correspondence between 3D shapes (typically modeled as Riemannian surfaces and discretized as triangular meshes). 
Ovsjanikov et al. \cite{ovsjanikov2012functional} introduced an elegant framework for describing shape correspondence as linear operators.

Let $\mathcal{X}$ and $\mathcal{Y}$ denote two surfaces between which we want to find correspondence, and by $L^2(\mathcal{X})$ and $L^2(\mathcal{Y})$ the respective functional spaces of scalar fields. We further assume to be given orthogonal bases on these spaces, denoted by $\{\phi_i \}_{i\geq 1} \subseteq L^2(\mathcal{X})$ and $\{\psi_i \}_{i\geq 1} \subseteq L^2(\mathcal{Y})$, respectively. 
{\em Functional correspondence} is a linear operator $T : L^2(\mathcal{X}) \rightarrow L^2(\mathcal{Y})$ mapping functions from one surface to another. It can be expressed w.r.t. to the given orthogonal bases as 
\begin{eqnarray}
Tf &=& \sum_{i\geq 1} \langle f, \phi_i \rangle_{L^2(\mathcal{X})} T \phi_i = \sum_{i\geq 1} \langle f, \phi_i \rangle_{L^2(\mathcal{X})} \sum_{j\geq 1} \langle T \phi_i , \psi_i \rangle_{L^2(\mathcal{Y})} \psi_j \nonumber \\
&=& \sum_{i,j \geq 1} c_{ij} \langle f, \phi_i \rangle_{L^2(\mathcal{X})} \psi_j. 
\end{eqnarray}
The coefficients $ c_{ij} = \langle T \phi_i , \psi_i \rangle_{L^2(\mathcal{Y})}$ encode the correspondence and can be thought of as a translation of Fourier coefficients between the two bases. 
Truncating the expansion at the first $k$ coefficients yields a compact finite-dimensional representation of the functional correspondence operator in the form of a $k\times k$ matrix $C = (c_{ij})$.

Given a set of $q\geq k$ corresponding functions $g_l \approx T f_l$, $l = 1,\hdots, q$, finding correspondence boils down to solving a linear system of $qk$ equations in $k^2$ variables, 
\begin{eqnarray}
\sum_{i \geq 1} c_{ij} \langle f_l, \phi_i \rangle_{L^2(\mathcal{X})} =  \sum_{j \geq 1}  \langle g_l, \psi_j \rangle_{L^2(\mathcal{Y})}  
\label{eq:funcorr}
\end{eqnarray}
relating the respective Fourier coefficients. 
In matrix form, the system~(\ref{eq:funcorr}) can be expressed as 
\begin{eqnarray}
AC = B,
\label{eq:funcorr1}
\end{eqnarray}
where $A = (\langle f_l, \phi_i \rangle_{L^2(\mathcal{X})} )$ and $B = (\langle g_l, \psi_i \rangle_{L^2(\mathcal{Y})} )$ are $q\times k$ matrices of Fourier coefficients of the given corresponding functions.

Eynard et al. \cite{Eynard2016} proposed a formulation of the problem where {\em two} functional maps $T_1 : L^2(\mathcal{X}) \rightarrow L^2(\mathcal{Y})$ and $T_2 : L^2(\mathcal{Y}) \rightarrow L^2(\mathcal{X})$ satisfying $T_1 T_2 = \mathrm{id}$ are considered simultaneously. In matrix representation, this amounts to having $C_1 C_2 = I$.  
Finding the two functional maps is thus formulated as an optimization problem on the biorthogonal manifold,
\begin{eqnarray}
\min_{(C_1, C_2) \in BO(k)} \| AC_1 - B \|^2_\mathrm{F} + \| A - B C_2 \|^2_\mathrm{F} + \rho(C_1, C_2), 
\label{eq:funcorr_biorth}
\end{eqnarray}
where $\rho(C_1, C_2)$ is a term adding some regularization on $C_1, C_2$.

We reproduced the results of \cite{Eynard2016} on a pair of human shapes with known groundtruth correspondence from the FAUST dataset \cite{bogo2014faust}. As data, we used the SHOT descriptor \cite{tombari2010unique} of dimension $q=320$. As the orthogonal bases $\{ \phi_i, \psi_i \}_{i=1}^k$, we used $k=30$ first eigenvectors of the discretized Laplace-Beltrami operator of the two shapes.  
As the baselines, we used the least-squares solution of~(\ref{eq:funcorr1}) as proposed by Ovsjanikov et al. \cite{ovsjanikov2012functional}, the method of Huang et al. \cite{huang2014functional} and the approximate solution of~(\ref{eq:funcorr_biorth}) using the penalty method as proposed by Eynard et al. \cite{Eynard2016}, where the optimization on the biorthogonal manifold is replaced by an unconstrained optimization with an additional penalty term $\alpha \| C_1 C_2 - I\|^2_\mathrm{F}$ (we used the penalty weight $\alpha = 10^6$).

We solved~(\ref{eq:funcorr_biorth}) using optimization on the biorthogonal manifold using conjugate gradient method.  
Following Eynard et al. \cite{Eynard2016}, we also included a penalty term in~(\ref{eq:funcorr_biorth}) promoting a funnel-shape structure of $C_1$ and $C_2$, 
$$
\rho(C_1, C_2) = \| C_1 \odot W \|_2 + \| C_2 \odot W \|_2,
$$
where $\odot$ denotes the element-wise (Hadamard) matrix product, and $W$ is a fixed matrix 
(for additional details, the reader is referred to \cite{Eynard2016}).

Figure~\ref{fig:Kim} evaluates the correspondence quality using the Princeton protocol \cite{kim2011blended}, depicting the percentage of correspondences falling within a certain error radius w.r.t the groundtruth. Higher curves represent better correspondence. 
%

\begin{figure}[!ht]
\begin{center}
\figurewidth=0.8\linewidth
\figureheight=7.5cm
%
%
\begin{tikzpicture}

\begin{axis}[%
width=\figurewidth,
height=\figureheight,
at={(0.758in,0.542in)},
scale only axis,
separate axis lines,
every outer x axis line/.append style={black},
every x tick label/.append style={font=\color{black}},
xmin=0,
xmax=0.5,
xtick={  0, 0.1, 0.2, 0.3, 0.4, 0.5},
xlabel={\small Geodesic error},
xmajorgrids,
every outer y axis line/.append style={black},
every y tick label/.append style={font=\color{black}},
ymin=0,
ymax=100,
ytick={  0,  20,  40,  60,  80, 100},
ylabel={\small \% Correspondences},
ymajorgrids,
axis background/.style={fill=white},
legend style={at={(0.97,0.03)},anchor=south east,legend cell align=left,align=left,draw=black}
]
\addplot [color=magenta,solid,line width=1.5pt]
  table[row sep=crcr]{%
0	1.10336817653891\\
0.01	10.5110336817654\\
0.02	25.8057491289199\\
0.03	39.5978513356562\\
0.04	50.2613240418118\\
0.05	56.4605110336818\\
0.06	59.9883855981417\\
0.07	62.4274099883856\\
0.08	64.2784552845528\\
0.09	66.0278745644599\\
0.1	67.167537746806\\
0.11	68.3797909407666\\
0.12	69.7662601626016\\
0.13	70.9494773519164\\
0.14	71.8931475029036\\
0.15	72.6190476190476\\
0.16	73.5264227642276\\
0.17	74.2523228803717\\
0.18	74.753193960511\\
0.19	75.2395470383275\\
0.2	75.6605691056911\\
0.21	76.154181184669\\
0.22	76.9236353077817\\
0.23	77.8745644599303\\
0.24	79.0577816492451\\
0.25	80.2627758420441\\
0.26	81.2354819976771\\
0.27	82.3896631823461\\
0.28	83.3188153310105\\
0.29	84.0447154471545\\
0.3	85.075493612079\\
0.31	85.9320557491289\\
0.32	86.4837398373984\\
0.33	86.955574912892\\
0.34	87.3983739837398\\
0.35	87.8339140534263\\
0.36	88.3420441347271\\
0.37	88.835656213705\\
0.38	89.4091173054588\\
0.39	90.0842044134727\\
0.4	90.6286295005807\\
0.41	91.3763066202091\\
0.42	92.0005807200929\\
0.43	92.5450058072009\\
0.44	92.9079558652729\\
0.45	93.2200929152149\\
0.46	93.3797909407666\\
0.47	93.6120789779326\\
0.48	93.8008130081301\\
0.49	93.9024390243902\\
0.5	93.9169570267131\\
0.51	93.931475029036\\
0.52	93.931475029036\\
0.53	93.9605110336818\\
0.54	93.9677700348432\\
0.55	94.0766550522648\\
0.56	94.2726480836237\\
0.57	94.4033101045296\\
0.58	94.4613821138211\\
0.59	94.5702671312427\\
0.6	94.7081881533101\\
0.61	94.8170731707317\\
0.62	95.0130662020906\\
0.63	95.1074332171893\\
0.64	95.383275261324\\
0.65	95.4776422764228\\
0.66	95.7099303135889\\
0.67	95.9349593495935\\
0.68	96.1309523809524\\
0.69	96.2761324041812\\
0.7	96.4793844367015\\
0.71	96.5882694541231\\
0.72	96.9149245063879\\
0.73	97.2343205574913\\
0.74	97.5972706155633\\
0.75	98.0473286875726\\
0.76	98.5627177700348\\
0.77	98.6280487804878\\
0.78	98.6353077816492\\
0.79	98.6353077816492\\
0.8	98.6570847851336\\
0.81	98.6570847851336\\
0.82	98.7441927990708\\
0.83	98.8240418118467\\
0.84	99.0490708478513\\
0.85	99.4918699186992\\
0.86	99.7314169570267\\
0.87	99.8620789779326\\
0.88	100\\
0.89	100\\
0.9	100\\
0.91	100\\
0.92	100\\
0.93	100\\
0.94	100\\
0.95	100\\
0.96	100\\
0.97	100\\
0.98	100\\
0.99	100\\
1	100\\
};
\addlegendentry{\footnotesize Ovsjanikov et al.};

\addplot [color=green,solid,line width=1.5pt]
  table[row sep=crcr]{%
0	3.94163763066202\\
0.01	26.0670731707317\\
0.02	44.5194541231127\\
0.03	53.4189895470383\\
0.04	60.5473286875726\\
0.05	65.8028455284553\\
0.06	69.3379790940767\\
0.07	71.7770034843206\\
0.08	73.7950058072009\\
0.09	75.4427990708479\\
0.1	76.8873403019744\\
0.11	78.1794425087108\\
0.12	79.1521486643438\\
0.13	80.4878048780488\\
0.14	81.5403600464576\\
0.15	82.585656213705\\
0.16	83.4131823461092\\
0.17	84.0592334494773\\
0.18	84.9157955865273\\
0.19	85.5110336817654\\
0.2	86.2369337979094\\
0.21	86.8321718931475\\
0.22	87.5943670150987\\
0.23	88.1678281068525\\
0.24	88.8937282229965\\
0.25	89.3002322880372\\
0.26	89.808362369338\\
0.27	90.1350174216028\\
0.28	90.5705574912892\\
0.29	90.9843205574913\\
0.3	91.3472706155633\\
0.31	91.7029616724739\\
0.32	91.9933217189315\\
0.33	92.3199767711963\\
0.34	92.6393728222996\\
0.35	93.0531358885017\\
0.36	93.4741579558653\\
0.37	93.8516260162602\\
0.38	94.178281068525\\
0.39	94.5993031358885\\
0.4	95.0421022067364\\
0.41	95.448606271777\\
0.42	95.7534843205575\\
0.43	96.0728803716609\\
0.44	96.4430894308943\\
0.45	96.6898954703833\\
0.46	96.8931475029036\\
0.47	97.0528455284553\\
0.48	97.1399535423926\\
0.49	97.2198025551684\\
0.5	97.2851335656214\\
0.51	97.3577235772358\\
0.52	97.4157955865273\\
0.53	97.5391986062718\\
0.54	97.6190476190476\\
0.55	97.7424506387921\\
0.56	97.8513356562137\\
0.57	98.0255516840883\\
0.58	98.1634727061556\\
0.59	98.23606271777\\
0.6	98.4901277584204\\
0.61	98.7151567944251\\
0.62	98.9547038327526\\
0.63	99.2813588850174\\
0.64	99.7604529616725\\
0.65	99.9491869918699\\
0.66	99.9709639953542\\
0.67	100\\
0.68	100\\
0.69	100\\
0.7	100\\
0.71	100\\
0.72	100\\
0.73	100\\
0.74	100\\
0.75	100\\
0.76	100\\
0.77	100\\
0.78	100\\
0.79	100\\
0.8	100\\
0.81	100\\
0.82	100\\
0.83	100\\
0.84	100\\
0.85	100\\
0.86	100\\
0.87	100\\
0.88	100\\
0.89	100\\
0.9	100\\
0.91	100\\
0.92	100\\
0.93	100\\
0.94	100\\
0.95	100\\
0.96	100\\
0.97	100\\
0.98	100\\
0.99	100\\
1	100\\
};
\addlegendentry{\footnotesize  Huang et al.};

\addplot [color=blue,solid,line width=1.5pt]
  table[row sep=crcr]{%
0	1.68408826945412\\
0.01	14.3583042973287\\
0.02	30.6983159117305\\
0.03	44.2871660859466\\
0.04	53.9779326364692\\
0.05	62.7831010452962\\
0.06	69.3960511033682\\
0.07	74.274099883856\\
0.08	78.0270034843206\\
0.09	80.4515098722416\\
0.1	82.3678861788618\\
0.11	84.0519744483159\\
0.12	85.2497096399535\\
0.13	86.1425667828107\\
0.14	86.6579558652729\\
0.15	87.1806039488966\\
0.16	87.4637049941928\\
0.17	87.8193960511034\\
0.18	88.0734610917538\\
0.19	88.2549361207898\\
0.2	88.4872241579559\\
0.21	88.6977351916376\\
0.22	88.9590592334495\\
0.23	89.0606852497096\\
0.24	89.3873403019744\\
0.25	89.7502903600465\\
0.26	90.0261324041812\\
0.27	90.2221254355401\\
0.28	90.3092334494773\\
0.29	90.4544134727062\\
0.3	90.6141114982578\\
0.31	90.6939605110337\\
0.32	90.8463995354239\\
0.33	90.9770615563298\\
0.34	91.144018583043\\
0.35	91.3400116144019\\
0.36	91.6594076655052\\
0.37	92.3127177700348\\
0.38	92.84262485482\\
0.39	93.5322299651568\\
0.4	94.3524970963995\\
0.41	95.1437282229965\\
0.42	95.8405923344948\\
0.43	96.5301974448316\\
0.44	97.147212543554\\
0.45	97.7061556329849\\
0.46	98.1779907084785\\
0.47	98.6498257839721\\
0.48	98.8966318234611\\
0.49	99.1869918699187\\
0.5	99.390243902439\\
0.51	99.455574912892\\
0.52	99.6225319396051\\
0.53	99.6660859465737\\
0.54	99.6806039488966\\
0.55	99.6806039488966\\
0.56	99.6806039488966\\
0.57	99.6806039488966\\
0.58	99.7096399535424\\
0.59	99.7459349593496\\
0.6	99.7459349593496\\
0.61	99.8040069686411\\
0.62	99.818524970964\\
0.63	99.8548199767712\\
0.64	99.9201509872242\\
0.65	99.9564459930314\\
0.66	99.9854819976771\\
0.67	100\\
0.68	100\\
0.69	100\\
0.7	100\\
0.71	100\\
0.72	100\\
0.73	100\\
0.74	100\\
0.75	100\\
0.76	100\\
0.77	100\\
0.78	100\\
0.79	100\\
0.8	100\\
0.81	100\\
0.82	100\\
0.83	100\\
0.84	100\\
0.85	100\\
0.86	100\\
0.87	100\\
0.88	100\\
0.89	100\\
0.9	100\\
0.91	100\\
0.92	100\\
0.93	100\\
0.94	100\\
0.95	100\\
0.96	100\\
0.97	100\\
0.98	100\\
0.99	100\\
1	100\\
};
\addlegendentry{\footnotesize Penalty};

\addplot [color=red,solid,line width=3pt]
  table[row sep=crcr]{%
0	1.29210220673635\\
0.01	8.13734030197445\\
0.02	21.4576074332172\\
0.03	34.4802555168409\\
0.04	46.6536004645761\\
0.05	56.8742740998839\\
0.06	64.8373983739837\\
0.07	70.9349593495935\\
0.08	75.5081300813008\\
0.09	79.8054587688734\\
0.1	82.6001742160279\\
0.11	85.7868757259001\\
0.12	89.5325203252033\\
0.13	92.6538908246225\\
0.14	93.6701509872242\\
0.15	94.359756097561\\
0.16	94.8098141695703\\
0.17	95.1582462253194\\
0.18	95.3106852497096\\
0.19	95.383275261324\\
0.2	95.4123112659698\\
0.21	95.4340882694541\\
0.22	95.4631242740999\\
0.23	95.4703832752613\\
0.24	95.4703832752613\\
0.25	95.4849012775842\\
0.26	95.4921602787456\\
0.27	95.5066782810685\\
0.28	95.5647502903601\\
0.29	95.7897793263647\\
0.3	96.0148083623693\\
0.31	96.2470963995354\\
0.32	96.5447154471545\\
0.33	96.9149245063879\\
0.34	97.32868757259\\
0.35	97.7497096399535\\
0.36	98.3812427409988\\
0.37	99.2668408826945\\
0.38	99.5572009291521\\
0.39	99.6878629500581\\
0.4	99.7822299651568\\
0.41	99.8040069686411\\
0.42	99.818524970964\\
0.43	99.8257839721254\\
0.44	99.8257839721254\\
0.45	99.8257839721254\\
0.46	99.8257839721254\\
0.47	99.8257839721254\\
0.48	99.8257839721254\\
0.49	99.8257839721254\\
0.5	99.8257839721254\\
0.51	99.8257839721254\\
0.52	99.8257839721254\\
0.53	99.8257839721254\\
0.54	99.8403019744483\\
0.55	99.8765969802555\\
0.56	99.9274099883856\\
0.57	99.9854819976771\\
0.58	99.9927409988386\\
0.59	100\\
0.6	100\\
0.61	100\\
0.62	100\\
0.63	100\\
0.64	100\\
0.65	100\\
0.66	100\\
0.67	100\\
0.68	100\\
0.69	100\\
0.7	100\\
0.71	100\\
0.72	100\\
0.73	100\\
0.74	100\\
0.75	100\\
0.76	100\\
0.77	100\\
0.78	100\\
0.79	100\\
0.8	100\\
0.81	100\\
0.82	100\\
0.83	100\\
0.84	100\\
0.85	100\\
0.86	100\\
0.87	100\\
0.88	100\\
0.89	100\\
0.9	100\\
0.91	100\\
0.92	100\\
0.93	100\\
0.94	100\\
0.95	100\\
0.96	100\\
0.97	100\\
0.98	100\\
0.99	100\\
1	100\\
};
\addlegendentry{\footnotesize Biorthogonal};

\end{axis}
\end{tikzpicture}%
\caption{\label{fig:Kim} \small Performance of different correspondence method on a pair of FAUST shapes evaluated using the Princeton
protocol. 
}
\end{center}
\end{figure}
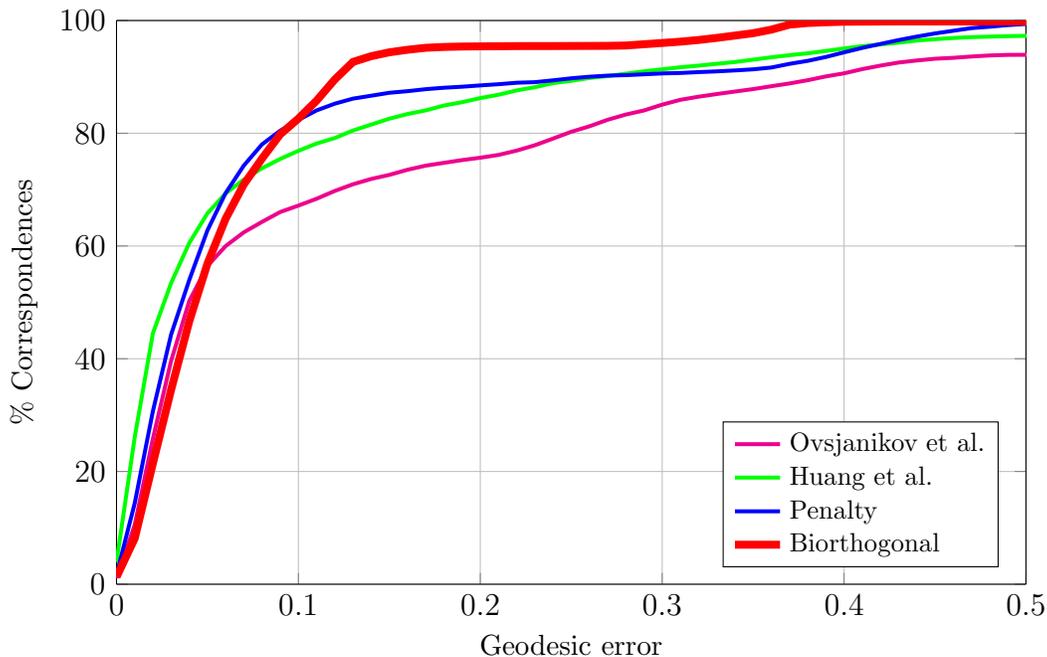

%
%

\section{Conclusions}

In this paper, we introduced and analyzed the biorthogonal manifold, allowing to efficiently perform optimization over pairs of orthogonal matrices. In future work, we will focus on applications where such problems arise.

\section*{Acknowledgement}
This research was supported by the ERC Starting Grant
No. 307047 (COMET). 

\bibliographystyle{plain}
\bibliography{main}

\end{document}